\documentclass[11pt]{amsart}
\usepackage{amssymb}
\usepackage{dsfont}

\newtheorem{theorem}{Theorem}[section]

\newtheorem{lemma}[theorem]{Lemma}
\newtheorem{proposition}[theorem]{Proposition}

\newtheorem{corollary}[theorem]{Corollary}

\newtheorem{definition}[theorem]{Definition}
\newtheorem{remark}[theorem]{Remark}

\newcommand{\N}{\mathbb N}

\newcommand{\R}{\mathbb R}

\newcommand{\F}{\mathcal{F}}
\newcommand{\mS}{\mathcal{S}}
\newcommand{\bK}{\mathbb K}

\author{Filip Strobin}

\address{Institute of Mathematics, Jan Kochanowski University in Kielce, ul. \'Swietokrzyska 15, 25-406 Kielce, Poland
}
\address{Institute of Mathematics, \L\'od\'z University of Technology, W\'olcza\'nska 215, 93-005
\L\'od\'z, Poland}
\email {filip.strobin@p.lodz.pl}

\author{Jaros\l aw Swaczyna}

\address{Institute of Mathematics, \L\'od\'z University of Technology, W\'olcza\'nska 215, 93-005
\L\'od\'z, Poland}
\email {jswaczyna@wp.pl}

\title[]{A code space for a generalized IFS}
\date{}

\begin{document}
\maketitle
\begin{abstract}
We study the concept of a code (or shift) space for a generalized iterated function system (GIFS in short). We prove that relations between GIFSs and their code spaces are analogous to the case of classical IFSs.  As an application, we consider the problem of connectedness of attractors of GIFSs. Many of our results are strengthenings of the ones proved recently by Mihail, Miculescu and Secelean, but some are completely new. 
\end{abstract}
\section{Introduction}
Let $(X,d)$ be a metric space and $m\in\N$.\\
By $X^m$ we denote the Cartesian product of $m$ copies of $X$. We consider it as a metric space with the maximum metric $d_m$:
$$
d_m((x_1,...,x_m),(y_1,...,y_m)):=\max\{d(x_1,y_1),...,d(x_m,y_m)\}.
$$
A function $f:X^m\rightarrow X$ is called a \emph{generalized $\varphi$-contraction} if for some nondecreasing, upper semicontinuous function $\varphi:[0,\infty)\rightarrow [0,\infty)$ with $\varphi(t)<t$ for $t>0$, the following condition holds   
$$
\forall_{x,y\in X^m}\;d(f(x),f(y))\leq \varphi(d_m(x,y)).
$$
Note that if the Lipschitz constant $Lip(f)<1$, then $f$ is a generalized $\varphi$-contraction. Also if $m=1$, then a generalized $\varphi$-contraction is called a $\varphi$-contraction (cf. \cite{Br}, \cite{JJ}).\\
By $\bK(X)$ we denote the metric space of all nonempty and compact subsets of $X$, endowed with the Hausdorff metric $H$:
$$
H(A,B):=\max\left\{\sup\left\{\inf\{d(x,y):y\in A\}:x\in B\right\},\sup\left\{\inf\{d(x,y):y\in B\}:x\in A\right\}\right\}
$$
\begin{definition}
\emph{If $f_1,...,f_n:X^m\rightarrow X$ are generalized $\varphi$-contractions, then the finite sequence $\mS=(f_1,...,f_n)$ is called }a generalized iterated function system of order $m$\emph{ (GIFS in short).}\\
\emph{If $\mS=(f_1,...,f_n)$ is a GIFS of order $m$, then by $F_\mS$ we denote the mapping $F_\mS:\bK(X)^m\rightarrow \bK(X)$ defined by:
$$
F_\mS(D_1,...,D_m):=f_1(D_1\times...\times D_m)\cup...\cup f_n(D_1\times...\times D_m)
$$}
\end{definition}
Clearly, GIFSs of order $1$ are classical iterated function systems (IFSs in short), which have been deeply considered in the last 30 years (see, a.e., \cite{H}, \cite{Ha}, \cite{B}).\\

In the recent papers \cite{Mi}, \cite{M}, \cite{M1}, \cite{MM1},\cite{MM2}, \cite{MS}, \cite{Se}, \cite{SS} and \cite{S}, the theory of GIFSs was developed. In particular, the following generalization of a classical result holds (\cite{M}, \cite{MM1}, \cite{MM2}, \cite{SS}):
\begin{theorem}\label{it2}
Assume that $(X,d)$ is a complete metric space and $\mS$ is a GIFS of order $m$. Then there exists a unique set $A_\mS\in\bK(X)$ such that
$$
F_\mS(A_\mS,...,A_\mS)=A_\mS.\label{111}
$$
Moreover, for every $D_1,...,D_m\in\bK(X)$, the sequence of iterates $(D_k)$ defined by
$$
D_{m+k}:=F_\mS(D_k,...,D_{k+m-1}),
$$
converges (with respect to the Hausdorff metric) to $A_\mS$.
\end{theorem}
Sets $A_\mS$ which are generated by GIFSs in the above sense will be called \emph{attractors} or \emph{fractals} (generated by GIFS $\mS$). We will always denote by $A_\mS$ the attractor of a GIFS $\mS$.\\

Also, in \cite{S} it is shown that GIFSs give us a new class of fractals, but does not give us everything:
\begin{theorem}
For any $m>1$ and $\alpha<1$, there exists a Cantor set $C(m,\alpha)\subset \R^2$ such that
\begin{itemize}
\item[(i)] there exists a GIFS $\mS=(f_1,...,f_4)$ on $\R^2$ of order $m$ such that $Lip(f_i)\leq \alpha$ and $C(m,\alpha)$ is the attractor of $\mS$;
\item[(ii)] $C(m,\alpha)$ is not an attractor of any GIFS $\mS$ on $\R^2$ of order $m-1$.
\end{itemize}
\end{theorem}
In particular, taking $m=2$, we get a Cantor subset of a plane which is an attractor of some GIFS and is not an attractor of any IFS.
\begin{theorem}
There is a Cantor set $C\subset \R^2$ which is not an attractor of any GIFS on $\R^2$.
\end{theorem}
In this paper we make a further study on GIFSs. We focus our attention on defining a code space for GIFSs 
 and on the problem of connectedness of attractors of GIFSs (these topics were partially considered in \cite{M1} and \cite{MS}). More precisely:\\
In Section $2$ we define and investigate a code (or shift) space for a GIFS. Note that in \cite{M1} (and \cite{MS}) there is another definition of a code space for GIFSs given. Our construction is a bit different (but equivalent) to that one but has the advantage that some further considerations can be handled in a bit technically simpler ways than in cited papers  (at least, we believe that it is so). Also, in \cite{M1} there is only the case $m=2$ considered.\\
Section $3$ is devoted to connections between GIFSs and their code spaces. 
Note that many results can be considered as extended counterparts of that given in \cite{M1} (we do not restrict to the case $m=2$ and to GIFSs consisting of Lipschitzian mappings), but some of them are completely new. In particular, we will prove a counterpart of the following one ($\{1,...,n\}^\N$ is the space of all sequences of elements from $\{1,...,n\}$):
\begin{theorem}\label{it1}
Assume that $\mS=(f_1,...,f_n)$ is an IFS on a complete metric space $X$. Then for every sequence $\alpha=(\alpha^1,\alpha^2,...)\in\{1,...,n\}^\N$, there is $x_\alpha\in X$ such that 
$$
\bigcap_{k\in\N}f_{\alpha^1}\circ...\circ f_{\alpha^k}(A_\mS)=\{x_\alpha\}.
$$
Moreover,
\begin{itemize}
\item[(i)] the mapping $\{1,...,n\}^\N\ni\alpha\rightarrow x_\alpha\in X$ is continuous;
\item[(ii)] $A_\mS=\{x_\alpha:\alpha\in\{1,....,n\}^\N\}$;
\item[(iii)] for every $\alpha=(\alpha^1,\alpha^2,...)\in\{1,...,n\}^\N$ and $D\in\bK(X)$, the sequence of compact sets $(f_{\alpha^1}\circ...\circ f_{\alpha^k}(D))_{k\in\N}$ converges to $\{x_\alpha\}$ (with respect to the Hausdorff metric). In particular,
\item[(iv)] for every $\alpha=(\alpha^1,\alpha^2,...)\in\{1,...,n\}^\N$ and every $x\in X$, $\underset{k\rightarrow \infty}{\lim}f_{\alpha^1}\circ...\circ f_{\alpha^k}(x)=x_\alpha$.
\end{itemize}
\end{theorem}
   
In Section $5$ we use the concepts developed in previous sections to study the problem of connectedness of attractors. Again, some of our results are extensions of that given in \cite{MS}.\\

At the end of this section let us state some denotations concerned with sequences.\\
If $\alpha$ is a finite sequence, then by $|\alpha|$ we denote the length of $\alpha$ (i.e., if $\alpha=(\alpha^1,...,\alpha^k)$, then $|\alpha|=k$; also, we set $|\emptyset|=0$). If $\alpha=(\alpha^1,...,\alpha^k)$ and $m\leq k$, then we denote $\alpha\vert_m=(\alpha^1,...,\alpha^m)$; also we set $\alpha\vert_0=\emptyset$.
If $\alpha$ is an infinite sequence, then $\alpha\vert_m$ has an analogous meaning. If $\alpha=(\alpha^1,...,\alpha^k)$, then we set $\alpha\hat\;\beta=(\alpha^1,...,\alpha^k,\beta)$, the concatenation of $\alpha$ and $\beta$. \\

\section{A generalized code space}
In this section we assume that $n,m\in\N$ are fixed.\\
At first, let us define $\Omega_1,\Omega_2,\ldots $ by the following inductive formula:\\
$$\Omega_1:=\{1,\ldots ,n\}$$
$$\Omega_{k+1}:=\underbrace{\Omega_k\times\ldots \times\Omega_k}_{m\mbox{ times}}\;\;\;\mbox{for }k\geq 1$$
Then for every $k\in\N$, let $${}_{k}\Omega:=\Omega_1\times \ldots \times\Omega_k$$ and define
$$\Omega_<:=\bigcup_{k\in\N}{}_{k}\Omega$$
and
$$\Omega:=\Omega_1\times \Omega_2\times\Omega_3\times\ldots$$
Note that the above definitions depend on $n$ and $m$, but we decided not to write it so the notation remains more clear.\\
The space $\Omega$ is called a \emph{code space}, and we consider it as a metric space with the natural metric (here $\alpha=(\alpha^1,\alpha^2,\ldots )$ and $\beta=(\beta^1,\beta^2,\ldots )$):
$$
d(\alpha,\beta):=\underset{i\in\N}{\sum}\frac{d_i(\alpha^i,\beta^i)}{(m+1)^i},
$$
where each $d_i$ is a discrete metric on $\Omega_i$. Clearly, $\Omega$ is a compact space (in fact, it is a Cantor space).
\begin{remark}
\emph{Let us note that while defining metric on $\Omega$, we could give it any of the form $d(\alpha,\beta)=\underset{i\in\N}{\sum}{q^{i}}{d_i(\alpha^i,\beta^i)}$, where $q\in(0,1)$. However, in view of Proposition $\ref{kontrtau}$, we need $q$ such that $q<\frac{1}{m}$. That is why we decided to choose $q=\frac{1}{m+1}$.}
\end{remark}
\begin{remark}\emph{
Note that in the case $m=1$, $$\Omega=\{1,...,n\}\times\{1,...,n\}\times\{1,...,n\}\times...$$ and it is the standard code space for IFS consisting of $n$ mappings.}
\end{remark}
\begin{remark}\emph{
In \cite{M1}, the code space is defined by:
$$
\Omega'=\{1,...,n_1\}\times\{1,...,n_2\}\times\{1,...,n_3\}\times...,
$$
where $n_k=n^{2^{k-1}}$, $k\geq 1$. Clearly, $n_k$ is the number of elements of $\Omega_k$ in the case when $m=2$ (this is the only case considered in \cite{M1}). On one hand our notation is more technically complicated, but on the other (we believe that) it is more natural in dealing with some further problems.
}\end{remark}
For $\alpha \in \Omega_<$, let us define $V(\alpha):=\{\beta \in \Omega:\alpha\prec\beta\}$ (the notation $\alpha\prec\beta$ means that the sequence $\beta$ is an extension of $\alpha$). It is an easy and well-known fact that the family $\{V(\alpha):\alpha \in\Omega_<\}$ is a base of the considered topology on $\Omega$.\\

If $k>1$ and $$\alpha=(\alpha^1,\ldots ,\alpha^k)=(\alpha^1,(\alpha^2_1,\ldots ,\alpha^2_m),\ldots ,(\alpha^k_1,\ldots ,\alpha^k_m))\in{}_{k}\Omega$$ then for any $i\in \{1,\ldots ,m\}$, we set
\begin{equation}
\alpha(i):=(\alpha^2_i,\alpha^3_i,\ldots ,\alpha^{k-1}_i).\label{n1}
\end{equation}
Clearly, $\alpha(i)\in {}_{k-1}\Omega$.\\
If $\alpha \in \Omega$, we define $\alpha(i) \in \Omega$ in an analogous way.\\
Now we will define a particular family of mappings. At first, define $\tau_1,\ldots ,\tau_n:\Omega^m\rightarrow\Omega$ in the following way (note that $\Omega^m$ is the Cartesian product of $m$ copies of $\Omega$): if $\alpha_1=(\alpha_1^1,\alpha_1^2,\alpha_1^3,\ldots),\ldots ,\alpha_m=(\alpha_m^1,\alpha_m^2,\alpha_m^3,\ldots)$, then set
$$
\tau_i(\alpha_1,\ldots ,\alpha_m):=(i,(\alpha_1^1,\alpha_2^1,\ldots ,\alpha_m^1),(\alpha_1^2,\alpha_2^2,\ldots ,\alpha_m^2),\ldots ).
$$
The following result (which is an extended counterpart of \cite[Lemma 3.2]{M1}) shows that $(\tau_1,...,\tau_n)$ can be considered as a canonical GIFS on $\Omega$.
\begin{proposition}\label{kontrtau}
$\mS_\Omega:=(\tau_1,\ldots ,\tau_n)$ is a GIFS on $\Omega$ such that $Lip(\tau_i) \leq \frac{m}{m+1}$ and $\Omega$ is its attractor.
\end{proposition}
\begin{proof}
Take any $i \in \{1,\ldots ,n\}$, and, for simplicity, set $\tau=\tau_i$. Let $\alpha=(\alpha_{1},\ldots ,\alpha_{m}), \beta=(\beta_{1},\ldots ,\beta_{m}) \in \Omega^{m}$. Then (we denote $\tau(\alpha)=((\tau(\alpha))^{1},(\tau(\alpha))^{2},\dots)$)

$$d(\tau(\alpha),\tau(\beta))=
\underset{j\in\N}{\sum}\frac{d_j((\tau(\alpha))^{j},(\tau(\beta))^{j})}{(m+1)^{j}}=
\frac{d_1((\tau(\alpha))^{1},(\tau(\beta))^{1})}{m+1} + \underset{j \geq 2}{\sum}\frac{d_j((\tau(\alpha))^{j},(\tau(\beta))^j)}{(m+1)^j}=$$

$$=\frac{d_{1}(i,i)}{m+1}+ \underset{j \in \N}{\sum}\frac{d_{j+1}((\tau(\alpha))^{j+1},(\tau(\beta))^{j+1})}{(m+1)^{j+1}}=
0+\frac{1}{m+1}\underset{j \in \N}{\sum}\frac{d_{j+1}((\tau(\alpha))^{j+1},(\tau(\beta))^{j+1})}{(m+1)^{j}}=$$


$$=\frac{1}{m+1}\underset{j \in \N}{\sum}\frac{d_{j+1}((\alpha_1^{j},\alpha_2^{j},\ldots ,\alpha_m^{j}),(\beta_1^j,\beta_2^j,\ldots ,\beta_m^j))}{(m+1)^{j}}=$$ $$=
\frac{1}{m+1}\underset{j \in \N}{\sum}\frac{max\{d_{j}(\alpha_1^{j},\beta_1^j),d_{j}(\alpha_2^{j},\beta_2^j),\ldots ,d_{j}(\alpha_m^{j},\beta_m^{j})\}}{(m+1)^{j}}\leq$$

$$\leq \frac{1}{m+1}\left(\underset{j\in \N}{\sum}\frac{d_{j}(\alpha_{1}^{j},\beta_{1}^{j})}{(m+1)^{j}} + \underset{j\in \N}{\sum}\frac{d_{j}(\alpha_{2}^{j},\beta_{2}^{j})}{(m+1)^{j}} + \ldots + \underset{j\in \N}{\sum}\frac{d_{j}(\alpha_{m}^{j},\beta_{m}^{j})}{(m+1)^{j}}\right) \leq$$

$$\leq \frac{m}{m+1}max\left\lbrace\underset{j \in \N}{\sum}\frac{d_{j}(\alpha_1^{j},\beta_1^{j})}{(m+1)^{j}},\underset{j \in \N}{\sum}\frac{d_{j}(\alpha_2^{j},\beta_2^j)}{(m+1)^{j}},\ldots ,\underset{j \in \N}{\sum}\frac{d_{j}(\alpha_m^{j},\beta_m^j)}{(m+1)^{j}}\right\rbrace=$$

$$=\frac{m}{m+1}max\{d(\alpha_{1},\beta_{1}),d(\alpha_{2},\beta_{2}),\ldots ,d(\alpha_{m},\beta_{m})\}=\frac{m}{m+1}d_m(\alpha,\beta)$$
We proved the first part of proposition. Now we will show that $A_{\mS_\Omega}=\Omega$. Clearly, $\underset{i \in \{1,...,m\}}{\bigcup}\tau_{i}(\Omega) \subset \Omega$.
Take any $\alpha=(\alpha^{1},\alpha^{2},...) \in \Omega$. It is easy to see that $\alpha=\tau_{\alpha^{1}}(\alpha(1),\alpha(2),\ldots ,\alpha(m))$, hence $\alpha \in \tau_{\alpha^{1}}(\Omega \times \ldots \times \Omega)\subset \underset{i\in \{1,\ldots n\}}{\bigcup}\tau_{i}(\Omega \times \ldots \times \Omega)$.
\end{proof}
Now we will define a certain family of mappings (we will use them in next parts of the paper). At first, let $\Pi_1,\Pi_2,\ldots $ be a family of sets defined by the following inductive formula:\\
$$\Pi_1:=\underbrace{\Omega\times\ldots \times\Omega}_{m\mbox{ times}}$$
$$\Pi_{k+1}:=\underbrace{\Pi_k\times\ldots \times\Pi_k}_{m\mbox{ times}},\;\;\mbox{ for }k\geq 1$$
For every $k\in\N$, we will define a family of mappings $\{\tau_\alpha:\Pi_k\rightarrow \Omega:\alpha\in{}_{k}\Omega\}$ by induction with respect to $k$. For $k=1$ we have already defined this family - this is just $\{\tau_1,\ldots ,\tau_n\}$. For $\alpha=(\alpha^1,\alpha^2,\ldots \alpha^{k+1})\in{}_{k+1}\Omega$, we set $\tau_\alpha:\Pi_{k+1}\rightarrow \Omega$ by
$$
\tau_\alpha(\beta_1,\ldots ,\beta_m):=\tau_{\alpha^1}(\tau_{\alpha(1)}(\beta_1),\ldots ,\tau_{\alpha(m)}(\beta_m)).
$$

\section{Generalized code space and GIFSs}
In this section we assume that we work with some fixed GIFS $\mS=(f_1,...,f_n)$ of order $m$, on a complete metric space $(X,d)$.\\ 
Symbols $\Omega_{k},\ {}_{k}\Omega,\;\Omega_<,\;\Omega$ keep their meaning from previous section.\\ 
At first, we define the family of spaces $X_{k}, \;k\in \N,$ by the following inductive formula:
$$X_{1}:=\underbrace{X\times\ldots \times X}_{m\mbox{ times}}$$
$$X_{k+1}:=\underbrace{X_{k}\times\ldots \times X_{k}}_{m\mbox{ times}},\;\;\;\mbox{for }k\geq1$$
We consider each $X_k$ as a metric space with maximum metric. It is easy to see that for any $k \in \N$, the space $X_{k}$ is isometric to $X^{m^k}$, endowed with the maximum metric.\\
We will define the family of functions $\{f_\alpha:X_{k}\rightarrow X:\alpha\in{}_{k}\Omega\}$ for each $k \in \N$ inductively (the induction is with respect to $k$). For $k=1$ this is just the family $\{f_1,...,f_n\}$. 
Assume that we have defined the functions $f_\alpha$ for $\alpha\in\;_k\Omega$. Then for every $\alpha=(\alpha^1,\ldots ,\alpha^k,\alpha^{k+1})\in{}_{k+1}\Omega$, set
$$
f_\alpha(x_1,\ldots ,x_m):=f_{\alpha^1}(f_{\alpha(1)}(x_1),\ldots ,f_{\alpha(m)}(x_m)).
$$
Where $(x_{1},...,x_m) \in X_{k}\times...\times X_k=X_{k+1}$.\\
\begin{remark}
\emph{Note that the family $\{\tau_{\alpha}:\alpha \in {}_{<}\Omega\}$ from previous section is constructed in the same way.} 
\end{remark}
\begin{remark}\emph{
Clearly, in the case of IFSs (i.e., in the case when $m=1$), $_k\Omega=\{1,...,n\}^k$ and if $\alpha=(\alpha^1,...,\alpha^k)\in \;_k\Omega$, then $f_\alpha=f_{\alpha^1}\circ...\circ f_{\alpha^k}$, hence introduced families of mappings are natural generalizations of such compositions. In the rest of this section we will show that, using this families, we can prove a counterpart of Theorem \ref{it1}. 
}\end{remark}
Now we switch our attention to the attractor $A_\mS$ of $\mS$. At first, we will define the family of sets $A_k$, $k\in\N$, by the following inductive formula:
$$A^\mS_{1}:=\underbrace{A_\mS\times\ldots \times A_\mS}_{m\mbox{ times}}$$
$$A^\mS_{k+1}:=\underbrace{A^\mS_k\times\ldots \times A^\mS_k}_{m\mbox{ times}},\;\;\;\mbox{for }k\geq 1$$
By an easy induction one can show that for each $i \in \N$, $A^\mS_{i} \subset X_{i}$, and $diam(A^\mS_{i})=diam(A_{S})$.\\
Now for every $k\in\N$ and every $\alpha\in\;_k\Omega$, define
\begin{equation}
A_\alpha:=f_\alpha(A^\mS_k).\label{333}
\end{equation}
By definition, $A_\mS=f_{1}(A_\mS\times\ldots \times A_\mS)\cup\ldots \cup f_{n}(A_\mS\times\ldots \times A_\mS)=A_{1}\cup \ldots \cup A_{n}$. It turns out that the following (extended counterpart of \cite[Theorem 3.1 (1) and (4)]{M1}) also holds (the case $m=1$ is easy and well known):
\begin{proposition}\label{p11}
For every $k\in\N$ and every $\alpha\in{}_{k}\Omega$,
$$
A_\alpha=\underset{{\beta\in\Omega_{k+1}}}{\bigcup}A_{\alpha\;\hat{ }\;\beta}.
$$
\end{proposition}
\begin{proof}
We will procede inductively. At first, fix any $i =1,\ldots ,n$. We have
$$\underset{{\beta\in\Omega_{2}}}{\bigcup}A_{i\;\hat{ }\;\beta} = \underset{{\beta\in\Omega_{2}}}{\bigcup}A_{(i,\beta)} = \underset{{\beta\in\Omega_{2}}}{\bigcup}f_{(i,\beta)}(A^\mS_2) = \underset{{(\beta_1,...,\beta_m)\in\Omega_{2}}}{\bigcup}f_{i}\left(f_{\beta_{1}}(A^\mS_1)\times\ldots \times f_{\beta_{m}}(A^\mS_1)\right) =$$

$$= f_{i}\left(\underset{{(\beta_1,...,\beta_m)\in\Omega_{2}}}{\bigcup}\left(f_{\beta_{1}}(A^\mS_1)\times\ldots \times f_{\beta_{m}}(A^\mS_1)\right)\right) = f_{i}\left(\underset{\beta_{1}\in \Omega_{1}}{\bigcup}\underset{\beta_{2}\in \Omega_{1}}{\bigcup}\ldots \underset{\beta_{m}\in \Omega_{1}}{\bigcup}\left(f_{\beta_{1}}(A^\mS_1)\times\ldots \times f_{\beta_{m}}(A^\mS_1)\right)\right) = $$

$$= f_{i}\left(\left(\underset{\beta_{1}\in \Omega_{1}}{\bigcup}f_{\beta_{1}}(A^\mS_1)\right)\times \left(\underset{\beta_{2}\in \Omega_{1}}{\bigcup}f_{\beta_{2}}(A^\mS_1)\right)\times \ldots \times \left(\underset{\beta_{m}\in \Omega_{1}}{\bigcup}f_{\beta_{m}}(A^\mS_1)\right)\right) =$$ $$= f_{i}(A_{\mS}\times \ldots \times A_{\mS}) = f_{i}(A^\mS_{1})=A_{i}$$

Now assume that thesis holds for some $k \in \N$.\\
Take any $\alpha=(\alpha^{1},\ldots ,\alpha^{k+1}) \in \;_{k+1}\Omega$. We have
$$
\underset{{\beta\in\Omega_{k+2}}}{\bigcup}A_{\alpha\;\hat{ }\;\beta} = \underset{{\beta\in\Omega_{k+2}}}{\bigcup}A_{(\alpha^{1},\ldots ,\alpha^{k+1},\beta)} = \underset{{\beta\in\Omega_{k+2}}}{\bigcup}f_{(\alpha^{1},\ldots ,\alpha^{k+1},\beta)}(A^\mS_{k+2}) =
$$
$$
=\underset{{\beta\in\Omega_{k+2}}}{\bigcup}f_{\alpha^1}\left(f_{(\alpha^1,\ldots ,\alpha^{k+1},\beta)(1)}(A^\mS_{k+1})\times\ldots \times f_{(\alpha^1,\ldots ,\alpha^{k+1},\beta)(m)}(A^\mS_{k+1})\right)=
$$
$$
=\underset{{(\beta_1,...,\beta_m)\in\Omega_{k+2}}}{\bigcup}f_{\alpha^1}\left(f_{(\alpha^{2}_{1},\ldots ,\alpha^{k+1}_{1},\beta_{1})}(A^\mS_{k+1})\times\ldots \times f_{(\alpha^{2}_{m},\ldots ,\alpha^{k+1}_{m},\beta_{m})}(A^\mS_{k+1})\right) =
$$
$$
=f_{\alpha^1}\left(\underset{{(\beta_1,...,\beta_m)\in\Omega_{k+2}}}{\bigcup}\left(f_{(\alpha^{2}_{1},\ldots ,\alpha^{k+1}_{1}, \beta_{1})}(A^\mS_{k+1})\times\ldots \times f_{(\alpha^{2}_{m},\ldots ,\alpha^{k+1}_{m},\beta_{m})}(A^\mS_{k+1})\right)\right)=
$$
$$
= f_{\alpha^1}\left(\underset{{\beta_{1}\in\Omega_{k+1}}}{\bigcup}\underset{{\beta_{2}\in\Omega_{k+1}}}{\bigcup}\ldots \underset{{\beta_{m}\in\Omega_{k+1}}}{\bigcup}\left(f_{(\alpha^{2}_{1},\ldots ,\alpha^{k+1}_{1}, \beta_{1})}(A^\mS_{k+1})\times\ldots \times f_{(\alpha^{2}_{m},\ldots ,\alpha^{k+1}_{m},\beta_{m})}(A^\mS_{k+1})\right)\right) =
$$
$$
= f_{\alpha^1}\left(\left(\underset{{\beta_{1}\in\Omega_{k+1}}}{\bigcup}f_{(\alpha^{2}_{1},\ldots ,\alpha^{k+1}_{1}, \beta_{1})}(A^\mS_{k+1})\right)\times\ldots \times \left(\underset{{\beta_{m}\in\Omega_{k+1}}}{\bigcup}f_{(\alpha^{2}_{m},\ldots ,\alpha^{k+1}_{m},\beta_{m})}(A^\mS_{k+1})\right)\right) = 
$$
$$
= f_{\alpha^1}\left(\left(\underset{{\beta_{1}\in\Omega_{k+1}}}{\bigcup}f_{(\alpha^{2}_{1},\ldots ,\alpha^{k+1}_{1})\;\hat{ }\;\beta_{1}}(A^\mS_{k+1})\right)\times\ldots \times \left(\underset{{\beta_{m}\in\Omega_{k+1}}}{\bigcup}f_{(\alpha^{2}_{m},\ldots ,\alpha^{k+1}_{m})\;\hat{ }\;\beta_{m}}(A^\mS_{k+1})\right)\right) =
$$
$$
= f_{\alpha^1}\left(\left(\underset{{\beta_{1}\in\Omega_{k+1}}}{\bigcup}A_{(\alpha^{2}_{1},\ldots ,\alpha^{k+1}_{1})\;\hat{ }\;\beta_{1}}\right)\times\ldots \times \left(\underset{{\beta_{m}\in\Omega_{k+1}}}{\bigcup}A_{(\alpha^{2}_{m},\ldots ,\alpha^{k+1}_{m})\;\hat{ }\;\beta_{m}}\right)\right) =
$$
$$
= f_{\alpha^1}\left(A_{(\alpha^{2}_{1},\ldots ,\alpha^{k+1}_{1})}\times\ldots \times A_{(\alpha^{2}_{m},\ldots ,\alpha^{k+1}_{m})}\right) =f_{\alpha^1}\left(A_{\alpha(1)}\times\ldots \times A_{\alpha(m)}\right)=
$$
$$
=f_{\alpha^1}\left(f_{\alpha(1)}(A^\mS_{k})\times\ldots \times f_{\alpha(m)}(A^\mS_{k})\right) = f_{\alpha}(A^\mS_{k+1}) = A_{\alpha} 
$$
\end{proof}
The following lemma will be usefull in further considerations.
\begin{lemma}\label{zwart}
Let $D \subset X$ be a bounded set and define sequence $(D_{k})_{k\in \N}$ by the following inductive formula:
$$D_{1}:=\underbrace{D\times\ldots \times D}_{m\mbox{ times}}$$
$$D_{k+1}:=\underbrace{D_k\times\ldots \times D_k}_{m\mbox{ times}},\;\;\;\mbox{for }k\geq 1$$
Then for any $\alpha=(\alpha^{1},\alpha^{2},\ldots ) \in \Omega$, $diam(f_{\alpha\vert_k}(D_{k}))\rightarrow 0 $.
\end{lemma}
\begin{proof}
Let $\varphi$ be a function which witness to the fact that $f_{1},...,f_{n}$ are generalized $\varphi$-contractions. We will show, by induction, that for any $k \in \N$, and any $\alpha\in\;_k\Omega$, $diam(f_{\alpha}(D_{k}))\leq \varphi^{(k)}(diam(D))$. 
For $i=1,...,n$, we have 
$$
diam(f_{i}(D_{1}))\leq \varphi(diam(D_{1}))=\varphi(diam(D))
$$
Now let us assume that for some $k \in \N$ and every $\alpha\in\;_k\Omega$, $diam(f_{\alpha}(D_{k}))\leq \varphi^{(k)}(diam(D))$. Then for $\alpha\in\;_{k+1}\Omega$,
$$diam(f_{\alpha}(D_{k+1}))=diam\left(f_{\alpha^1}\left(f_{\alpha(1)}(D_{k})\times\ldots \times f_{\alpha(m)}(D_{k})\right)\right)\leq
$$
$$
\leq \varphi\left(diam\left(f_{\alpha(1)}(D_{k})\times\ldots \times f_{\alpha(m)}(D_{k})\right)\right)=$$
$$
=\varphi\left( max\left\lbrace diam\left(f_{\alpha(1)}(D_{k})\right),\ldots, diam\left(f_{\alpha(m)}(D_{k})\right)\right\rbrace \right)\leq$$ $$\leq \varphi\left( max\left\lbrace\varphi^{(k)}(diam(D)),\ldots ,\varphi^{(k)}(diam(D))\right\rbrace \right)=\varphi^{(k+1)}(diam(D)).$$
In particular, for every $\alpha\in \Omega$, $diam(f_{\alpha\vert_k}(D_k))\leq \varphi^{(k)}(diam(D))$.
As $\varphi^{(k)}(t)\rightarrow 0$ for any $t \geq 0$, our proof is complete.
\end{proof}
The above results shows that first part of a generalization of Theorem \ref{it1} (which is also an extended counterpart of \cite[Theorem 3.1 (2) and (3)]{M1}) holds:
\begin{proposition}\label{jedg}
For every $\alpha\in\Omega$, the sequence $(A_{\alpha \vert _{k}})_{k \in \N}$ is decreasing and $diam(A_{\alpha\vert_k})\rightarrow 0$. In particular, there exists $x_\alpha\in X$ such that $\underset{k\in\N}{\bigcap}A_{\alpha\vert_{k}}=\{x_\alpha\}$.
\end{proposition}
\begin{proof}
By Proposition \ref{p11}, the sequence $(A_{\alpha\vert_k})$ is decreasing. Hence it is enough to show that $diam(A_{\alpha\vert_k})\rightarrow 0$, but this follows from Lemma $\ref{zwart}$ by setting $D=A_{\mS}$.
\end{proof}
By the above result, we can define the mapping $g:\Omega\rightarrow X$ by
$$
g(\alpha):=x_\alpha,
$$
where $x_\alpha$ is the unique point of $\underset{k\in\N}{\bigcap}A_{\alpha\vert_{k}}$.
\begin{lemma}\label{l1}
Let $(K_k)$ be a decreasing sequence of compact subsets of $X$ such that $\bigcap_{k\in\N}K_k=\{x\}$ for some $x\in X$. Then for every open set $U$ with $x\in U$, there is $k\in\N$ such that $K_k\subset U$.
\end{lemma}
\begin{proof}
Since for every $y\notin U$ there exists $k\in\N$ such that $y\notin K_k$, we have that
$$
K_1\setminus U\subset\bigcup_{k\in\N}K_1\setminus K_k
$$
Since sets $K_1\setminus K_k$ are open in $K_1$ and $K_1\setminus U$ is compact, there are $k_1,...,k_r\in\N$ such that
$$
K_1\setminus U\subset\bigcup_{i=1}^rK_1\setminus K_{k_i}.
$$
It is enough to take $k=\max\{k_1,...,k_r\}$.
\end{proof}
The next result shows that a counterpart of $(i)$ and $(ii)$ of Theorem \ref{it1} holds (this is also an extension of \cite[Theorem 3.1 (7)]{M1}).
\begin{theorem}\label{tt11}
The following conditions hold
\begin{itemize}
\item[(i)] The mapping $g$ is continuous;
\item[(ii)] $g(\Omega)=A_\mS$.
\end{itemize}
\end{theorem} 
\begin{proof}
(i) Let $\alpha \in \Omega$ and $U\subset \Omega$ be open, such that $g(\alpha)\in U$. As $\{g(\alpha)\}=\underset{k\in\N}{\bigcap}A_{\alpha\vert_{k}}$ and $(A_{\alpha\vert_{k}})_k,$ is decreasing sequence of compact sets, Lemma \ref{l1} implies that there is $k_{0} \in \N$ such that $A_{\alpha\vert_{k_{0}}} \subset U$. Hence for $\beta \in V(\alpha\vert_{k_{0}})$, we have that $g(\beta) \in A_{\beta\vert_{k_{0}}}=A_{\alpha\vert_{k_{0}}}\subset U$. This ends the proof of $(i)$.\\
(ii) Take any $\alpha \in \Omega$. Then $g(\alpha)\in \underset{k\in\N}{\bigcap}A_{\alpha\vert_{k}} \subset A_{\alpha^{1}} \subset A_{\mS}$. We proved that $g(\Omega)\subset A_{\mS}$. Now let $x \in A_{\mS}$. We have to construct proper $\alpha \in \Omega$ such that $g(\alpha)=x$. By Proposition $\ref{jedg}$, it is enough to prove $x \in A_{\alpha\vert_{k}}$ for any $k \in \N$.  We will proceed inductively. Since $A_{\mS}=\underset{i=1,...,n}{\bigcup}A_{i}$, there exists some $\alpha^1 \in\{1,...,n\}= \Omega_{1}$ such that $x \in A_{\alpha^1}$. 
Now let us assume that for some $k \in \N$ we defined $\alpha^{1}\in\Omega_{1},\ldots,\alpha^{k}\in\Omega_{k}$ such that for any $l \leq k$, we have $x\in A_{(\alpha^{1},\ldots,\alpha^{l})}$. By Proposition $\ref{p11}$, there exists $\alpha^{k+1} \in \Omega_{k+1}$ such that $x \in A_{(\alpha^1,...,\alpha^{k+1})}$. Finally set $\alpha = (\alpha^{1}, \alpha^{2},...)$. Then $x \in A_{\alpha\vert k}$ for any $k \in \N$ and our proof is finished. 
\end{proof}
Finally, we show that counterparts of $(iii)$ and $(iv)$ of Theorem \ref{it1} holds.
\begin{theorem}\label{t11}
For every $\alpha\in\Omega$ and every closed and bounded set $D\subset X$, $f_{\alpha\vert_k}(D_{k})\overset{H}{\rightarrow}\{g(\alpha)\}$, where the sequence $(D_k)$ is defined as in Lemma \ref{zwart} (let us note that Hausdorff metric may be used for closed and bounded sets).
\end{theorem}
\begin{proof}
Define 
$$
D':=D\cup A_\mS ,\ D'_{1}=\underbrace{D'\times\ldots \times D'}_{m\mbox{ times}},
$$

$$
D'_{k+1}:=\underbrace{D'_{k}\times \ldots \times D'_{k}}_{m\mbox{ times}},\ for\ k\geq 1.
$$ 
Then for every $k\in\N$, 
$$g(\alpha)\in A_{\alpha\vert_k}=f_{\alpha\vert_k}(A_{k})\subset f_{\alpha\vert_k}(D'_{k}).$$
On the other hand, by Lemma $\ref{zwart}$, $diam(f_{\alpha\vert_k}(D'_{k}))\rightarrow 0$ and $\underset{k\in\N}{\bigcap} f_{\alpha\vert_k}(D'_{k})=\{g(\alpha)\}$. In particular, for every $r>0$ there is $k_0\in\N$ such that for $k\geq k_0$, $diam(f_{\alpha\vert_k}(D_k')<r$. In particular, $f_{\alpha\vert_k}(D_k)\subset f_{\alpha\vert_k}(D_k')\subset B(g(\alpha),r)$, where $B(g(\alpha),r)$ states for the open ball.

\end{proof}
If $\mathbf{x}=(x_1,\ldots ,x_m)\in X^m$, then we define the sequence $(\mathbf{x}_k)$ by the following inductive way:
$$
\mathbf{x}_1:=\mathbf{x}
$$
$$
\mathbf{x}_{k+1}:=(\mathbf{x}_{k},\ldots ,\mathbf{x}_{k}),\;\;\;\mbox{for }k\geq 1 
$$
Clearly, $\mathbf{x}_k\in X_k$ for every $k\in\N$.
\begin{corollary}\label{t12}
For every $\mathbf{x}\in X^m$ and $\alpha\in\Omega$, we have that $\lim_{k\rightarrow\infty}f_{\alpha\vert_{k}}(\mathbf{x}_{k})=g(\alpha)$.
\end{corollary}
\begin{proof}
It simply follows from Theorem $\ref{t11}$. Take $D:=\{x_{1},\ldots ,x_{m}\}$. Then for each $k \in \N$, $\mathbf{x}_k\in D_{k}$, where $D_{k}$ is as in this Theorem. In particular, $f_{\alpha\vert_{k}}(\mathbf{x}_k)\in f_{\alpha\vert_{k}}(D_k)$ and the result follows.
\end{proof}

For the next result we need some further notation. If $k\geq 1$ and $$x=(x_1,...,x_{m^k},x_{m^k+1},...,x_{2m^k},...,x_{(m-1)m^k+1},...,x_{m^{k+1}})\in X^{m^{k+1}},$$ then for every $i\in \{1,...,m\}$,  we set $x(i):=(x_{(i-1)m^k+1},...,x_{im^k})$, the $i$-th block of $x$. For every $k\in\N$, let $h_k:X^{m^k}\rightarrow X_k$ be the "most natural" bijection, i.e.,
$$
h_1(x_1,...,x_m):=(x_1,...,x_m)
$$
and, for $x\in X^{m^{k+1}}$, 
$$
h_{k+1}(x):= (h_{k}(x(1)),h_{k}(x(2)),\ldots , h_{k}(x(m))).
$$
Clearly, each $h_k$ is an isometry. 
The next result is an extended counterpart of \cite[Theorem 3.1 (5)]{M1} and a part of \cite[Theorem 3.1 (6)]{M1}.
\begin{proposition}
For every $k\in\N$ and every $\alpha\in\;_k\Omega$, there exists a unique $x_\alpha\in X$ such that $f_\alpha(h_k(x_\alpha,...,x_\alpha))=x_\alpha$. Moreover, $x_\alpha\in A_\alpha$ for every $\alpha\in \Omega_<$ and $\{x_\alpha:\alpha\in\Omega_<\}$ is dense in $A_\mS$.
\end{proposition}
\begin{proof}
We first show the existence and uniqueness  of $x_\alpha$. For every $k\in\N$ and every $\alpha\in\;_k\Omega$, let $\tilde{f}_\alpha:=f_\alpha\circ h_k$. Then $\tilde{f}_\alpha:X^{m^{k}}\rightarrow X$. Recall that we consider $X^{m^{k}}$ as metric space with maximum metric.  It is enough to show that each $\tilde{f}_\alpha$ is a generalized $\varphi$-contraction - then we can use the fixed point theorem for generalized $\varphi$-contractions proved in \cite[Theorem 3.1(i)]{SS}.\\
Let $\varphi$ be a function which witness to the fact that $f_1,...,f_n$ are generalized $\varphi$-contractions. We will prove that for any $k\in \N$ and any $\alpha\in\;_k\Omega$,
\begin{equation}
\forall_{x,y\in X^{m^k}}\;\;d(\tilde{f}_\alpha(x),\tilde{f}_\alpha(y))\leq\varphi^{(k)}(d_m(x,y)).\label{444}
\end{equation}
Chose any $i\in \{1,...,n\}$. We have for any $x,y\in X^m$,
$$
d(\tilde{f}_i(x),\tilde{f}_i(y))=d(f_i(x),f_i(y))\leq \varphi(d_m(x,y)).
$$
Hence the case $k=1$ is true. Now assume that (\ref{444}) is true for some $k\ge1$. For $\alpha\in\;_{k+1}\Omega$, we have for any $x,y\in X^{m^{k+1}}$ (we denote the metrices on $X$, $X^{m^k}$ and on $X^{m^{k+1}}$ by the same letter $d$),
$$
d(\tilde{f}_\alpha(x),\tilde{f}_\alpha(y))=d\left(f_\alpha(h_{k+1}(x)),f_\alpha(h_{k+1}(y))\right)=$$ $$=
d\left(f_{\alpha^1}\left(f_{\alpha(1)}(h_k(x(1))),...,f_{\alpha(m)}(h_k(x(m)))\right),f_{\alpha^1}\left(f_{\alpha(1)}(h_k(y(1))),...,f_{\alpha(m)}(h_k(y(m)))\right)\right)=
$$
$$
=d\left(f_{\alpha^1}\left(\tilde{f}_{\alpha(1)}(x(1)),...,\tilde{f}_{\alpha(m)}(x(m))\right),f_{\alpha^1}\left(\tilde{f}_{\alpha(1)}(y(1)),...,\tilde{f}_{\alpha(m)}(y(m))\right)\right)\leq
$$
$$
\leq \varphi\left(\max\left\{d\left(\tilde{f}_{\alpha(1)}(x(1)),\tilde{f}_{\alpha(1)}(y(1))\right),...,d\left(\tilde{f}_{\alpha(m)}(x(m)),\tilde{f}_{\alpha(m)}(y(m))\right)\right\}\right)\leq
$$
$$
\leq \varphi\left(\max\left\{\varphi^k(d(x(1),y(1))),...,\varphi^{(k)}(d(x(m),y(m)))\right\}\right)=\varphi(\varphi^{(k)}(d(x,y)))=\varphi^{(k+1)}(d(x,y)).
$$
Now we show that the element $x_\alpha\in A_\alpha$ for every $\alpha\in\Omega_<$. Clearly, $\mS_{A_\mS}:=\left(f_1\vert_{A_\mS},...,f_n\vert_{A_\mS}\right)$ is a GIFS, hence, by the uniqueness of $x_\alpha$, we see that $x_\alpha\in A_\mS$. But then for $\alpha\in\;_k\Omega$, we have
$$
x_\alpha=\tilde{f}_\alpha(x_\alpha,...,x_\alpha)=f_\alpha(h_k(x_\alpha,...,x_\alpha))\in f_\alpha(A^\mS_k)=A_\alpha.
$$
Finally, let $x\in A_\mS$. Then $x=g(\alpha)$ for some $\alpha\in\Omega$. Let $U$ be any open set containing $x$. Since $\{x\}=\bigcap_{k\in\N}A_{\alpha\vert_k}$, by Lemma \ref{l1}, there is $k\in\N$ such that $A_{\alpha\vert_k}\subset U$. But then $x_{\alpha\vert_k}\in A_{\alpha\vert_k}\subset U$.
\end{proof}
The next result will complete the picture of GIFSs. Again, the case $m=1$ is known (and the case $m=2$ and $k=1$ is an extended counterpart of \cite[Theorem 3.1 (8)]{M1}).
Define the family of mappings $\{g_{k}:\Pi_k \rightarrow X_{k}:k\in\N\}$ by the following inductive formula:
$$g_1(\alpha_1,\ldots ,\alpha_m):=(g(\alpha_1),\ldots ,g(\alpha_m))$$
$$g_{k+1}(\beta_{1},\ldots,\beta_{m}):=(g_{k}(\beta_{1}),\ldots,g_{k}(\beta_{m})),\;k\geq 1$$
\begin{theorem}\label{diag2}
 For every $k\in \N$ and $\alpha \in {}_{k}\Omega$, $f_{\alpha}\circ g_{k} = g \circ \tau_{\alpha}$.
 \end{theorem} 
 \begin{proof}
 We will precede inductively with respect to $k$. Let $k=1$ and $\alpha=(\alpha_{1},\ldots ,\alpha_{m})\in \Omega\times\ldots \times\Omega$. Let $(\mathbf{x}_{k})$ is a sequence built as before Corollary \ref{t12} for some arbitrary taken $\mathbf{x}\in X^{m}$. Then:
$$f_i\circ g_1(\alpha)=f_i \left( g({\alpha_{1}}),\ldots ,g({\alpha_{m}})\right) =f_{i}\left(\underset{k \rightarrow \infty}{lim}(f_{\alpha_{1}\vert_{k}}(\mathbf{x}_{k})),\ldots ,\underset{k \rightarrow \infty}{lim}(f_{\alpha_{m}\vert_{k}}(\mathbf{x}_{k}))\right)=$$
$$\underset{k \rightarrow \infty}{lim}\left(f_{i}(f_{\alpha_{1}\vert_{k}}(\mathbf{x}_{k}),\ldots ,f_{\alpha_{m}\vert_{k}}(\mathbf{x}_{k}))\right)=\underset{k \rightarrow \infty}{lim}(f_{\tau_{i}(\alpha)\vert_{k+1}}(\mathbf{x}_{k+1}))=g(\tau_{i}(\alpha))=g\circ \tau_i(\alpha).$$
Now assume that for some $k \in \N$ we have the thesis and take any $\beta \in {}_{k+1}\Omega$ and some $\alpha=(\alpha_{1},\ldots,\alpha_{m})\in \Pi_{k+1}$.\\
Then (we denote $\beta=(\beta^{1},\ldots,\beta^{k+1})):$
$$
f_{\beta}\circ g_{k+1}(\alpha) = f_{\beta}(g_{k}(\alpha_{1}),\ldots,g_{k}(\alpha_{m})) = f_{\beta^{1}}\left( f_{\beta(1)}(g_{k}(\alpha_{1})),\ldots,f_{\beta(m)}(g_{k}(\alpha_{m}))\right) = 
$$ 
$$
 = f_{\beta^{1}} \left( (f_{\beta(1)}\circ g_{k})(\alpha_{1}),\ldots,(f_{\beta(m)}\circ g_{k}) (\alpha_{m}) \right) = f_{\beta^{1}} \left( (g\circ \tau_{\beta(1)}) (\alpha_{1}),\ldots,(g \circ \tau_{\beta(m)})(\alpha_{m}) \right) = 
$$ 
$$
= f_{\beta^{1}} \left( g(\tau_{\beta(1)} (\alpha_{1})),\ldots,g(\tau_{\beta(m)}(\alpha_{m})) \right) = \left( f_{\beta^{1}}\circ g_{1} \right) \left( \tau_{\beta(1)} (\alpha_{1}),\ldots,\tau_{\beta(m)} (\alpha_{m}) \right) =
$$
$$
\left( g \circ \tau_{\beta^{1}} \right) \left( \tau_{\beta(1)} (\alpha_{1}),\ldots,\tau_{\beta(m)} (\alpha_{m}) \right) = g \left( \tau_{\beta^{1}}(\tau_{\beta(1)} (\alpha_{1}),\ldots,\tau_{\beta(m)} (\alpha_{m}) \right) = g(\tau_{\beta}(\alpha))=g \circ \tau_{\beta}(\alpha)
$$
\end{proof}

Finally, we give an extension of counterpart of \cite[Theorem 3.1 (9)]{M1}. We say that a GIFS $\mS$ is \emph{totally disconnected}, if for every $k\in\N$ and every distinct $\alpha,\beta\in\;_k\Omega$, $A_\alpha\cap A_\beta=\emptyset$. We skip an easy proof of following proposition, which gives us important informations about totally disconnectedness of $A_{\mS}$.
\begin{proposition}
A GIFS $\mS$ is totally disconnected iff the mapping $g$ is injective.
\end{proposition}
\begin{proof}
At first note that if $x\in A_\beta$ for some $\beta\in\Omega_<$, then $x=g(\alpha)$ for some $\alpha\in\Omega$ such that $\beta\prec \alpha$ - this can be shown similarly as in the proof of Theorem \ref{tt11} (ii).\\
Hence if for some distinct $\beta_1,\beta_2\in \;_k\Omega$, there is $x\in A_{\beta_1}\cap A_{\beta_2}$, then for some $\alpha_1,\alpha_2\in\Omega$ with $\beta_1\prec\alpha_1$ and $\beta_2\prec\alpha_2$, $x=g(\alpha_1)=g(\alpha_2)$. In particular, $g$ is not injective. The converse implication is obvious.
\end{proof}

\section{Application - connectedness of attractors of GIFSs}
In this section we will study the problem of connectedness of attractors of GIFSs. Main theorems will be implied by some general and abstract results connected with properties of some compact spaces which admit certain families of compact sets.\\
A family of finite sequences $\mathcal{F}$ is called a \emph{tree}, if for every $\alpha=(\alpha_1,...,\alpha_k)\in \F\setminus\{\emptyset\}$, $\alpha\vert_{k-1}\in \F$.\\
A tree $\F$ is called a \emph{pruned tree}, if for every $\alpha\in\F$, there is an element $\beta$ such that $\alpha\hat\;\beta\in \F$.\\
A tree $\F$ is called a \emph{finitely splitting tree}, if for every $\alpha\in\F$ there is only finitely many elements $\beta$ such that $\alpha\hat\;\beta\in\F$.\\
If $\mathcal{F}$ is a tree, then an infinite sequence $\alpha$ is called a \emph{node of} $\mathcal{F}$, if for every $k\in\N$, $\alpha\vert_k\in \mathcal{F}$.\\
Let $X$ be a topological space and $\mathcal{F}$ be a pruned and finitely splitting tree. A family $\{D_\alpha:\alpha\in \mathcal{F}\}$ of {\bf compact} subsets of $X$ is called a \emph{proper family (adjusted to} $\mathcal{F}$), if
\begin{itemize}
\item[(a)] $X=D_\emptyset$;
\item[(b)] for every $\alpha\in \mathcal{F}$, $D_\alpha=\underset{\{ \beta:\ \alpha \hat \;\beta \in \mathcal{F} \} }{\bigcup}D_{\alpha \hat \; \beta}$;
\item[(c)] for every node $\alpha$ of $\mathcal{F}$, $\bigcap_{k\in\N}D_{\alpha\vert_k}$ is a singleton.
\end{itemize}
If $X$ is a topological space and $\{D_\alpha:\alpha\in\mathcal{F}\}$ is a proper family of subsets of $X$, then for every node $\alpha$ of $\mathcal{F}$, by $x_\alpha$ we denote the only element of $\bigcap_{k\in\N}D_{\alpha\vert_k}$. Clearly, for every $x\in X$, there is a node $\alpha$ of $\mathcal{F}$ such that $x=x_\alpha$.\\
It turnes out that the notion of proper families is appropriate for our study:
\begin{proposition}
Assume that $A$ is an attractor of some GIFS. Then $A$ admits a proper family of sets. Moreover, if $A$ is connected, then $A$ admits a proper family consisting of connected sets.
\end{proposition}
\begin{proof}
Let $(f_1,...,f_n)$ be a GIFS such that $A=f_1(A\times...\times A)\cup...\cup f_n(A\times...\times A)$.\\
For every $\alpha\in \Omega^<$, let $A_\alpha$ be defined as in (\ref{333}). Additionally, let $A_\emptyset=A$. Since $f_\alpha$ are continuous, then each $A_\alpha$ is compact, and if $A$ connected, then it is also connected. Then Proposition \ref{jedg} implies that the family $\{A_\alpha:\alpha\in\Omega^<\}$ satisfies our needs.
\end{proof}
We skip an easy proof of the following
\begin{lemma}\label{l111}
Let $(X,d)$ be a metric space and $(D_k)$ be a decreasing family of compact sets such that $\bigcap_{k\in\N} D_k$ is a singleton. Then $diam(D_k){\rightarrow} 0$.
\end{lemma}
\begin{lemma}\label{cor11}
Let $(X,d)$ be a metric space and $\{D_\alpha:\alpha\in\mathcal{F}\}$ be a proper family of subsets of $X$. Then $\max\{diam(D_\alpha):\alpha\in\mathcal{F},\;|\alpha|=k\}\underset{k\rightarrow\infty}{\longrightarrow }0$.
\end{lemma}
\begin{proof}
For every $k\in\N$, let $a_k=\max\{diam(D_\alpha):\alpha\in\mathcal{F},\;|\alpha|=k\}$. Clearly, $(a_k)$ is nonincreasing. Assume on contrary that $(a_k)$ does not converge to $0$ and let $\varepsilon>0$ be such that $a_k\geq \varepsilon$ for every $k\in\N$. By induction, we will define a node of $\mathcal{F}$ for which $diam(D_{\alpha\vert_k})\geq \varepsilon$ for every $k \in \N$. By our assumption (and since $\mathcal{F}$ is finitely splitting), there is $\alpha^1$, such that the family $\{\alpha\in\mathcal{F}:\alpha\vert_1=\alpha^1\mbox{ and }diam(D_\alpha)\geq \varepsilon\}$ is infinite. Assume that for some $k$, we defined $\alpha^1,...,\alpha^k$ such that $\{\alpha\in\mathcal{F}:\alpha\vert_k=(\alpha^1,...,\alpha^k)\mbox{ and }diam(D_\alpha)\geq \varepsilon\}$ is infinite. Then there is $\alpha_{k+1}$ such that the family $\{\alpha\in\mathcal{F}:\alpha\vert_{k+1}=(\alpha^1,...,\alpha^{k+1})\mbox{ and }diam(D_\alpha)\geq \varepsilon\}$ is infinite. In this way we define the sequence $\alpha=(\alpha^1,\alpha^2,...)$ which is a node of $\mathcal{F}$ and such that $diam(D_{\alpha\vert_k})\geq \varepsilon$ for $k\in\N$. We get a contradiction with Lemma \ref{l111}.
\end{proof}
Recall that a metric space $(X,d)$ has the \emph{property $S$}, if for every $\epsilon>0$, $X$ can be covered by a finite family of connected sets with diameter $<\epsilon$. It is known that a connected and compact metric space $X$ is locally connected iff it has the property $S$ (cf. \cite[8.2, 8.4]{N}).
\begin{theorem}\label{m11}
Let $(X,d)$ be a connected metric space which admits a proper family of subsets of $X$ consisting of connected sets. Then $X$ is locally connected.
\end{theorem}
\begin{proof}
 Let $\{D_\alpha:\alpha\in\mathcal{F}\}$ be a proper family of subsets of $X$ consisting of connected sets. For every $k\in\N$, let $$\mathcal{A}_k:=\{D_\alpha:\alpha\in\mathcal{F},\;|\alpha|=k\},$$
and 
$$a_k:=\max\{diam(D):D\in\mathcal{A}_k\}.$$
By Lemma \ref{cor11}, $a_k\rightarrow 0$. Moreover, each $\mathcal{A}_k$ is a finite cover of $X$ consisting of connected sets. In particular, $X$ has the property $S$ and the result follows.
\end{proof}
As a direct implication, we get
\begin{theorem}
Assume that $A$ is an attractor of some GIFS. If $A$ is connected, then it is locally connected.
\end{theorem}

Recall that a topological $X$ space is \emph{arcwise connected} if for every $x,y\in X$ there is a continuous function $f:[0,1]\rightarrow X$ such that $f(0)=x$ and $f(1)=y$.\\
A sequence of sets $(D_1,...,D_m)$ is called a \emph{chain}, if $D_i\cap D_{i+1}\neq\emptyset$ for $i=1,...,m-1$. If $(D_1,...,D_m)$ is a chain and $x\in D_1$ and $y\in D_m$, then $(D_1,...,D_m)$ is called a \emph{chain which connects $x$ and $y$}.  A family of sets $\mathcal{A}$ is \emph{connected}, if for every $A,B\in\mathcal{A}$, there is a chain $(D_1,...,D_m)$ such that $D_1=A$ and $D_m=B$ and $D_i\in \mathcal{A}$ for $i \in \{1, \ldots, m\}$.\\
The following Lemma was proved in \cite{MS}. We say that a set $\Delta=\{y_0,...,y_m\}$ is a division of $[0,1]$, if $0=y_0<y_1<...<y_m=1$ (we will always use the convention that elements of divisions are listed in this way). We also denote $||\Delta||:=\max\{y_{i+1}-y_i:i=0,1,...,m-1\}$.
\begin{lemma}\label{l123}
Let $(X, d)$ be a complete metric space and $(a_k)_{k\geq 0}$ be a sequence
of positive numbers convergent to 0. Let $(\Delta_k)_{k\geq 0}$ be a sequence of divisions
of the unit interval $[0, 1]$ (denote $\Delta_k=\{y^k_0,...,y^k_{l_k}\}$) such that $\Delta_{k}\subset \Delta_{k+1}$ and $\lim_{k\rightarrow \infty}||\Delta_k||=0$. Let $(g_k)$ be a sequence of functions such that $g_k:\Delta_k\rightarrow X$, $g_{k+1}\vert_{\Delta_k}=g_k$ and for every $m\geq k$, every $y^k_i$ and every $y_j^m\in[y^k_i,y^k_{i+1}]\cap \Delta_m$, 
$$
d\left(g_m \left( y^m_j \right) ,g_k\left(y^k_i\right)\right)\leq a_k\;\;\mbox{   and   }\;\;d\left(g_m \left( y^m_j \right) ,g_k\left(y^k_{i+1}\right)\right)\leq a_k.
$$
Then there exists a continuous function $g:[0,1]\rightarrow X$ such that $g\vert_{\Delta_k}=g_k$ for every $k\in\N$.
\end{lemma}

\begin{theorem}\label{mt1}
Let $(X,d)$ be a metric space which admits a proper family $\{D_\alpha:\alpha\in\mathcal{F}\}$ of subsets of $X$ such that for every $\alpha\in\mathcal{F}$, the family $\{D_{\alpha\hat\;\beta}:\alpha\hat\;\beta\in \mathcal{F}\}$ is connected. Then $X$ is arcwise connected.
\end{theorem}


\begin{proof}
 Let $\{D_\alpha:\alpha\in\mathcal{F}\}$ be a proper family of subsets of $X$ such that for every $\alpha\in \F$, the family
$$
\mathcal{D}_\alpha:=\{D_{\alpha\hat\;\beta}:\alpha\hat\;\beta\in\mathcal{F}\}
$$
is connected. 
To prove that $X$ is arcwise connected, chose $x,y\in X$. We will use Lemma \ref{l123} - we will construct appropriate sequences of reals, divisions and partial functions.\\
For every $k\in\N$, let $$\mathcal{A}_k:=\{D_\alpha:\alpha\in\mathcal{F},\;|\alpha|=k\},$$
and let $a_k:=\max\{diam(D):D\in\mathcal{A}_k\}$. By Lemma \ref{cor11}, $a_k\rightarrow 0$.\\
Define the sequence of divisions $(\Delta_k)$ and the sequence of functions $(g_k)$ by induction such that the following conditions holds for every $k\in\N$ and $i=0,...,l_k-1$ (we will always abbreviate elements of $\Delta_k$ by $y^k_0,...,y^k_{l_k}$ so  that $y^k_i<y^k_{i+1}$):
\begin{itemize}
\item[(i)] $\Delta_0=\{0,1\}$, $g_0(0)=x$ and $g_0(1)=y$;
\item[(ii)] $\Delta_{k}\subset \Delta_{k+1}$;
\item[(iii)] the elements of $\Delta_{k+1}\cap [y^k_i,y^k_{i+1}]$ are uniformly distributed on the interval $[y^k_i,y^k_{i+1}]$;
\item[(iv)] the set $\Delta_{k+1}\cap [y^k_i,y^k_{i+1}]$ has at least three elements;
\item[(v)] $g_k:\Delta_k\rightarrow X$;
\item[(vi)] there is $D\in\mathcal{A}_k$ such that for every $m\geq k$, $g_m(\Delta_m\cap[y^k_i,y^k_{i+1}])\subset D$;
\item[(vii)] if $k\geq 1$, then $g_{k}\vert_{\Delta_{k-1}}=g_{k-1}$.
\end{itemize}
Let $\Delta_0$ and $g_0$ be as in $(i)$. Assume that we have already defined $\Delta_0,\Delta_1,...,\Delta_n$ and $g_0,...,g_n$ for some $n\geq 0$ so that $(i)-(vii)$ are satisfied for every $k=0,...,n$ and $i=0,...,l_k-1$ (of course, in $(vi)$ we can only consider $n\geq m\geq k$).\\
Now let $i \in \{0,...,l_{k}-1 \}$. By $(vi)$, there is $D\in\mathcal{A}_k$ such that $g_k(y^k_i),g_k(y^k_{i+1})\in D$. By our assumptions, there is a chain $(D_1,....,D_n)$ which connects $g_k(y^k_{i})$ and $g_k(y^k_{i+1})$ such that $n\geq 2$ (clearly, we can assume $n\geq 2$), $D_i\subset D$ and $D_i\in\mathcal{A}_{k+1}$. Let $y_0:=y^k_{i}$, $y_{n}:=y^k_{i+1}$ and $y_1,..,y_{n-1}$ be uniformly distributed on $[y_0,y_n]$, and let $x_0:=g_k(y^k_{i})$, $x_n:=g_k(y^k_{i+1})$ and $x_j\in D_j\cap D_{j+1}$ for $j=1,...,n-1$. Then define $g(y_j):=x_j$ for $j=0,...,n$.\\
Let $\Delta_{k+1}$ be a family of all $y_j's$ which appear in this construction (for all $i's$), and $g_{k+1}$ be the union of all partial functions $g$.\\
Proceeding inductively, we get desired sequences. Now it is easy to check that the assumptions of Lemma \ref{l123} are satisfied, hence there exsists a continuous function $g:[0,1]\rightarrow X$ such that $g(0)=x$ and $g(1)=y$. This ends the proof.
\end{proof}

\begin{lemma}\label{lemat1}
Let $(X,d)$ be a connected metric space and $\mathcal{A}$ be a finite family of closed and nonempty subsets of $X$ such that $X=\bigcup\mathcal{A}$. Then $\mathcal{A}$ is connected.
\end{lemma}
\begin{proof}
Let $D\in \mathcal{A}$ and set
$$
\mathcal{A}_D:=\{(D_1,...,D_n):n\geq 2,\;D=D_1,\;D_i\cap D_{i+1}\neq\emptyset,\;D_i\in\mathcal{A}\}
$$
It is enough to show that for every $E\in \mathcal{A}$ there is $(D_1,...,D_n)\in\mathcal{A}_D$ such that $D_n=E$. Assume on the contrary that it is not the case and let $E\in\mathcal{A}$ be a set which witness to this. Set $$\mathcal{B}:=\{G:\exists_{(D_1,...,D_n)\in \mathcal{A}_D}\;\exists_{i=1,...,n}\;G=D_i\},$$
 $F:=\bigcup\mathcal{B}$ and $H:=\bigcup(\mathcal{A}\setminus\mathcal{B})$. By our assumption, $F$ and $H$ are nonempty ($H$ is nonempty, because $E\notin \mathcal{B}$). They are also closed and disjoint. This is a contradiction.
\end{proof}
Finally, Theorem \ref{mt1} and Lemma \ref{lemat1} imply the following generalization of \cite[Theorem 3.1]{MS}
\begin{theorem}\label{coraaa}
Assume that $A$ is an attractor of some GIFS $\mS=(f_1,...,f_n)$. Then the following conditions are equivalent:
\begin{itemize}
\item[(i)] $A$ is connected;
\item[(ii)] $A$ is arcwise connected;
\item[(iii)] the family $\{f_1(A\times...\times A),...,f_n(A\times...\times A)\}$ is connected.
\end{itemize}
\end{theorem}

\begin{proof}
Implication $(ii)\Rightarrow(i)$ is trivial and implication $(i)\Rightarrow (iii)$ follows from Lemma \ref{lemat1}. Now assume that the family $\{A_{1},...,A_{n}\}$ is connected (we use the notation from (\ref{333})). In view of Theorem \ref{mt1}, it is enough to show that for every $k\in\N$ and every $\alpha\in\;_k\Omega$, the family $\{A_{\alpha\hat\;\beta}:\beta\in\Omega_{k+1}\}$ is connected. We will prove it by induction.\\ 
Let $i\in \{1,...,n\}$. By our assumption, the family $\{A_{j}: j=1,...,n\}$ is connected. This implies that the family
$$
\left\{ A_{j_1}\times...\times A_{j_m}:j_1,...,j_m \in \{1,...,n\} \right\}
$$
is connected, and hence the family $\{f_{i}(A_{j_1}\times...\times A_{j_m}):j_1,...,j_m=1,...,n\}$ is connected. But this gives thesis for $k=1$ since $\{f_{i}(A_{j_1}\times...\times A_{j_m}):j_1,...,j_m \in \{1,...,n \} \}=\{A_{i\hat\;\beta}:\beta\in \Omega_2\}$.\\
Assume that we proved the thesis for some $k$ and let $\alpha\in\;_{k+1}\Omega$.
By inductive assumption, families $\{A_{\alpha(1)\hat\;\beta}:\beta\in\Omega_{k+1}\}$,...,$\{A_{\alpha(m)\hat\;\beta}:\beta\in\Omega_{k+1}\}$ are connected, and this easily implies that the family
$$
\{A_{\alpha(1)\hat\;\beta_1}\times...\times A_{\alpha(m)\hat\;\beta_m}:\beta_1,...,\beta_m\in\Omega_{k+1}\}
$$
is also connected. Similarly as in the proof of Proposition \ref{p11}, we can show that for any $\beta=(\beta_1,...,\beta_m)\in \Omega_{k+2}$,
$$
A_{\alpha\hat\;\beta}=f_{\alpha^1}(A_{\alpha(1)\hat\;\beta_1}\times...\times A_{\alpha(m)\hat\;\beta_m}).\label{ab1}
$$
All in all, we get that $\{A_{\alpha\hat\;\beta}:\beta\in\Omega_{k+2}\}$ is connected.
\end{proof}



\end{document}